\documentclass{amsart}

\usepackage{amsfonts,amssymb,enumerate,color,bm,hhline,makecell,array}

\usepackage{array}
\usepackage{mathtools,amssymb}
\usepackage{mathrsfs}
\usepackage{comment}
\usepackage[all,ps,cmtip]{xy} 
\usepackage[normalem]{ulem}
\usepackage{hyperref}
\usepackage{tikz-cd}
\usepackage{amscd}
\usepackage[shortlabels]{enumitem}
\usepackage{tensor}
\usepackage{tikz}
\usetikzlibrary{matrix,arrows}
\usepackage{bm}
\usepackage{graphicx}
\usepackage{extarrows}
\usepackage{subcaption}
\usepackage{mathtools}
\usepackage{comment}
\usetikzlibrary{cd, calc}

\addtocontents{toc}{\protect\setcounter{tocdepth}{1}}

\newtheorem{theorem}{Theorem}
\newcommand{\To}{\longrightarrow}

\newtheorem{lemma}{Lemma}[section]
\newtheorem{defi}{Definition}[section]
\newtheorem{coro}{Corollary}
\newtheorem{corollary}[lemma]{Corollary}

\newtheorem{remark}[lemma]{Remark}

\newcommand{\Jacbar}{\overline{J}}

\newtheorem{thm}[lemma]{Theorem}

\usepackage{mathrsfs}

\newcommand{\len}{\operatorname{length}}

\title{Stable maps, multiplicities, and compactified Jacobians}
\author{Yifan Zhao}
\address{Department of Mathematics, Imperial College, London SW7 2AZ, United Kingdom}
\email{yz23@ic.ac.uk}
\usepackage[
  backend=biber,
  style=numeric,
]{biblatex}

\begin{document}

\maketitle
\begin{abstract}
Let $C$ be a complex projective integral curve with planar singularities. In this note, we study numerical relations among its versal deformation space, moduli space of stable maps, and compactified Jacobian. In particular, we correct a statement by Fantechi--G\"ottsche--van Straten on the multiplicity of the $\delta$-constant stratum of the versal deformation space at $[C]$. We also give a necessary and sufficient condition for the original claim to hold. 
\end{abstract}
\section{Introduction}
Let $C$ be a complex projective integral curve with planar singularities, i.e. the embedding dimension at any point of $C$ is at most 2. Denote by $g_a$ and $g$ the arithmetic and geometric genera of $C,$ respectively. We study numerical relations among three fundamental spaces associated with $C$:
the (smooth) space $B$ of its versal deformation \cite{def}, the 0-dimensional moduli space $M_{g}(C,[C])$ of genus $g$ stable maps \cite{Kon}, and its compactified Jacobian $\Jacbar_C$ \cite{ak,sou}.

The $\delta$-invariant $\delta(C)$ of the curve is by definition $g_{a}-g.$ Inside a versal deformation space $B$ of $C$, we consider the reduced closed subscheme $ B^{\delta}\subset B$ parametrising curves with the same $\delta$-invariant as $\delta(C).$ This closed subscheme $B^{\delta}$ is called the $\delta$-constant stratum of $B$ in \cite{eu}.

Let $m([C], B^{\delta})$ be the Hilbert--Samuel multiplicity of $B^{\delta}$ at the base point $[C]$. Generalizing Beauville's result \cite[Proposition 2.2]{Beau} on the Euler characteristic $\chi(\Jacbar_C)$ of compactifed Jacobians for irrational curves, Fantechi--G\"ottsche--van Straten \cite[Theorem 1]{eu} prove that
\begin{align}\label{eq 1}
    \chi(\Jacbar_{C}) = \begin{cases}
            0 &g>0,\\
            m([C], B^{\delta}) &g=0.
           \end{cases}
\end{align}
See Remark \ref{rmk 2.1} for identification of the Hilbert--Samuel multiplicity with the notion of multiplicity used in \cite{eu}.

It is also stated in \cite[Theorem 2]{eu} that for the integral planar curve $C$, 
\begin{equation} \label{eq *} \tag{*}
m([C], B^{\delta})=\len(M_{g}(C,[C])).
\end{equation}
However, Oblomkov--Yun \cite[Section 7.2]{oy} later provided examples of planar rational curves $C$ for which $$\chi(\Jacbar_{C})>\len(M_{0}({C},[{C}])),$$ thereby  disproving (\ref{eq *}). Our first result shows that (\ref{eq *}) holds if and only if the fat point $M_{g}(C,[C])$ has a local complete intersection singularity.
\begin{theorem}[Theorem \ref{thm 3.1}]\label{theorem 1}
The multiplicity of the point $[C]$ in $B^{\delta}$ satisfies:
$$
m([C],B^{\delta})\geq \len(M_{g}(C,[C]),
$$ and the equality holds if and only if $M_{g}(C,[C])$ is a local complete intersection.  
\end{theorem}

We moreover prove a corrected version of (\ref{eq *}), equating $m([C], B^{\delta})$ to the length of a similar 0-dimensional moduli space of stable maps: 

\begin{theorem}[Theorem \ref{thm 3.2}]\label{thm 2}
 Let $\mathcal{C}^{T}\to T$ be a flat family of integral planar curves with central fibre $\mathcal{C}^{T}_0=C$. Suppose that $T$ and the relative compactified Jacobian $\Jacbar(\mathcal{C}^{T}/T)$ are smooth, and that all fibres of $\mathcal{C}^{T}/T$ other than $\mathcal{C}^{T}_0=C$ have geometric genus greater than $g$. Then 
\begin{equation*} 
m([C], B^{\delta})=\len(M_{g}(\mathcal{C}^{T},[C]))\geq\len(M_{g}(C,[C])).
\end{equation*}    
\end{theorem}
For a fixed curve $C$, such a family $\mathcal{C}^{T}/T$ always exists by \cite{eu}. When $C$ is moreover rational, Theorem \ref{thm 2} and the equality (\ref{eq 1}) give
\begin{equation} \label{eq 2}
 \chi(\Jacbar_C)=\len(M_{g}(\mathcal{C}^{T},[C])).   
\end{equation}
\subsection*{Rational curves on K3 surfaces}
The topology of compactified Jacobians plays an important rôle in algebraic geometry, knot theory, and geometric representation theory (see, e.g.,  \cite{GKS,ORS,shen} and the references therein). In particular, the Euler characteristic  $\chi(\Jacbar_C)$ is closely related to the classical Yau--Zaslow formula \cite{YZ} of rational curves on K3 surfaces, which motivated the works  \cite{Beau,eu}.

For $C$ a rational curve on a K3 surface $S$, we can take the base $T$ in Theorem \ref{thm 2} to be an open subset of the linear system $|C|$. Denoting by $f_C$ the composition $\mathbb{P}^1\to C\xhookrightarrow{}S,$  we obtain:
\begin{coro}[Corollary \ref{coro 3.1}]\label{coro 1}
Suppose $C$ is a rational curve on a K3 surface $S.$ The moduli space of stable maps $M_0(S,[C])$ contains a connected component $M_0(S,[C])_{[f_C]}$ with a unique closed point $[f_C].$ We have
\begin{equation}\label{eq 3}
\chi(\Jacbar_C)=\len(M_0(S,[C])_{[f_C]}).
\end{equation}
\end{coro}
The right hand side of (\ref{eq 3}) can be viewed as the local contribution of the rational curve $C$ to the genus 0 reduced Gromov--Witten invariant \cite{bl,MP} of $S$. Indeed, suppose that all curves in the linear system $|C|$ are integral, the genus 0 reduced Gromov--Witten invariant can be expressed as
\begin{align*}
\text{GW}_0(S,[C])&=\len(M_0(S,[C])) \\ 
&=
\sum_{C_0\in |C| \text{ is rational}}\len(M_0(S,[C])_{[f_{C_0}]}) \\
&=\sum_{C_0\in |C| \text{ is rational}}\chi(\Jacbar_{C_0}),
\end{align*}
 where the first equality follows from the fact that the reduced GW virtual dimension equals the actual dimension of $M_0(S,[C]),$ and the last equality is Corollary \ref{coro 1}.

By Beauville's arguments \cite[Corollary 2.3]{Beau}, the above shows that $\text{GW}_0(S,[C])$ satisfies the Yau--Zaslow formula. While (\ref{eq 3}) is also stated in \cite[Theorem 2]{eu}, the proof there requires the incorrect equality (\ref{eq *}). We complete the proof by replacing (\ref{eq *}) with Theorem \ref{thm 2}.

As part of Theorem \ref{thm 2}, we obtain an intuitive description of $\len(M_0(S,[C])_{[f_C]})$, which is the local contribution to $\text{GW}_0(S,[C])$: the linear system $|C|$ (analytically or étale locally) admits a  classifying map to the versal deformation space $B$ 
$$
|C|\To B.
$$ Denote by $i([C],|C|\cdot B^{\delta};B)$ the intersection multiplicity \cite[Definition 7.1]{ful} of $|C|$ and $B^{\delta}$ at $[C]\in B.$ Then $$
\len(M_0(S,[C])_{[f_C]})=i([C],|C|\cdot B^{\delta};B).
$$ 

\subsection*{Idea of the proof}Our approach is intersection-theoretic in nature and builds heavily on \cite{eu}. The proofs of Theorem \ref{theorem 1} and \ref{thm 2} rely crucially on the various interpretations of the multiplicity $m([C], B^{\delta})$. We review in Section \ref{sec 2} the necessary background on multiplicities, and prove Theorems \ref{theorem 1} and \ref{thm 2} in Section \ref{sec 3}. 
\subsection*{Acknowledgements}
I thank Richard Thomas for suggesting that I think about \cite{eu} and for crucial discussions, suggestions and support. I am also grateful to  Soheyla Feyzbakhsh for many helpful discussions, and to Francesca Carocci, Thomas Dedieu and Oscar Kivinen for useful email correspondence.

This work was supported by the Engineering and Physical Sciences Research
Council [EP/S021590/1], the EPSRC Centre for Doctoral Training in Geometry and Number
Theory (The London School of Geometry and Number Theory), University College London.

\section{Multiplicities}\label{sec 2}
We first recall the definition of the Hilbert--Samuel multiplicity of a (fat) point:
\begin{defi}[Hilbert--Samuel multiplicity]
Given a closed point $x$ in an $n$-dimensional variety $X$, suppose $I$ is an ideal of the local ring $\mathcal{O}_{X,x}$ such that the radical $\sqrt{I}$ equals $m_x,$ the maximal ideal of $\mathcal{O}_{X,x}.$

For large $i$, $\len(\mathcal{O}_{X,x}/I^i)
$ is a degree $n$ polynomial in $i$ and we can write  $$
\len(\mathcal{O}_{X,x}/I^i)=e(I)\cdot\frac{i^n}{n!}+\text{lower degree terms}.
$$ The integer $e(I)$ is the Hilbert--Samuel multiplicity of the ideal $I.$

When $I=m_x,$ $e(I)$ is called the multiplicity of $x\in X,$ which we denote by $m(x,X).$
\end{defi}
The multiplicity $m(x,X)$ of $x\in X$ can be interpreted geometrically as an intersection multiplicity \cite[Definition 7.1]{ful}. Recall that if two subvarieties $M,N\subset\mathbb{A}^m$ intersect only at $x\in \mathbb{A}^m$, then the intersection multiplicity $i(x,M\cdot N; \mathbb{A}^{m})$ is defined as the intersection product $M\cdot N\in A_0(x)\cong \mathbb{Z}$ and is at most $\len(\mathcal{O}_{M\times_{\mathbb{A}^m}N})$. Equality holds if the local ring $\mathcal{O}_{M\times_{\mathbb{A}^m}N}$ is Cohen-Macaulay, while in general only $i(x,M\cdot N; \mathbb{A}^{m})$ is invariant under deformations.
\begin{lemma}[{\cite[Example 12.4.5]{ful}}]\label{lem 2.1}
Embed the n-dimensional variety $X$ Zariski locally inside an affine space $\mathbb{A}^{m}$. Let $L$ be an $m-n$ dimensional subvariety of $\mathbb{A}^{m}$ such that $x$ is an isolated closed point in the intersection $X\cdot L,$ then 
$$
m(x,X)=\emph{min } i(x,X\cdot L; \mathbb{A}^{m}),
$$ where the minimum is taken over all such subvarieties $L\subset \mathbb{A}^{m}.$ 
\end{lemma}

\begin{remark}\label{rmk 2.1}
One may also define $m(x,X)$ via dynamic intersections, as in \emph{\cite[p.71, 75]{ag}} and \emph{\cite[Chapter 11]{ful}}. The dynamic definition is the one used in \emph{\cite{eu}}, and is central to the proof of \emph{\cite[Theorem 1]{eu}}.

More precisely, embed the $n$-dimensional variety $X$ Zariski locally inside an affine space $\mathbb{A}^{m}.$ Let $L$ be a generic $m-n$ dimensional linear subspace of $\mathbb{A}^{m}$ passing through $x.$

Fix a small open (analytic) neighbourhood $U$ of $x\in \mathbb{A}^{m}$ and then deform $L$ to a generic nearby linear subspace $L'$ (which need not contain $x$). Inside $U,$ the subvarieties $L'$ and $X$ intersect transversally at finitely many points.

The number $\#(X\cap L'\cap U)$ is independent of the embedding into the affine space. It is proved in  \emph{\cite[A.15 on p.123]{ag}} that $\#(X\cap L'\cap U)=m(x,X)$. 
\end{remark}

Given a $0$-dimensional fat point $Z\subset X$ and the corresponding ideal $I_Z\subset \mathcal{O}_{X,Z^{\text{red}}},$ the Hilbert--Samuel multiplicity $e(I_Z)$ can be computed via a Segre class \cite[Section 4]{ful}.
\begin{lemma}[{\cite[Example 4.3.4)]{ful}}]\label{lem 2.2}
Consider the Segre class $$s(Z,X)\in A_{*}(Z)=A_0(Z)=\mathbb{Z}[Z^{\emph{red}}]$$ of the closed subscheme $Z\subset X.$ We have $$
e(I_Z)=s(Z,X).
$$ 
\end{lemma}
    Segre classes satisfy the following fundamental birational invariance property, which will be important to us.

\begin{lemma}[{\cite[Proposition 4.2(a)]{ful}}]\label{2}
Let $f:X' \to X$ be a proper birational morphism between irreducible varieties. Suppose $Z\subset X$ is a closed subscheme of $X,$ denote by $g$ the induced morphism $$
f^{-1}(Z)\xrightarrow[]{g}Z,
$$
we have:
$$
g_*\bigl(s(f^{-1}(Z),X')\bigr)=s(Z,X).
$$
\end{lemma}

\section{Stable maps, multiplicities, and compactified Jacobians}\label{sec 3}
\subsection{Locally versal families}
We start by recalling the notion of locally versal families, which are convenient to work with in our situation.

Let $c_1,\cdots,c_n$ be the singularities of the curve $C$. For each $i$, consider the (smooth) miniversal deformation space\footnote{Also called the semi-universal deformaion space.} $B(c_i)$ of the singularity $c_i$ (see, e.g., \cite{DH}). Let $\mathcal{C}'\to S$ be a flat family of integral planar curves such that $S$ is smooth and the central fibre $\mathcal{C}'_0=C.$
\begin{defi}[Locally versal family, {\cite[Definition 4.1]{gv}}]
The family $\mathcal{C}'/S$ is called locally versal at $0\in S$ if the induced tangent map$$
T_0S\To \prod_{i=1}^{n}T_{[c_i]}B(c_i)$$
is surjective. Equivalently, the classifying map $$
S\xrightarrow[]{\text{ j }} \prod_{i=1}^{n}B(c_i),
$$ which exists étale locally, is a smooth morphism around $0\in S$.

The family $\mathcal{C}'/S$ is called locally versal if it is locally versal at all points $s\in S.$
\end{defi}
Denote by $B^{\delta_i}(c_i)$ the $\delta$-constant locus of $B(c_i),$  which is the reduced closed subscheme that parametrises deformations of the singularity $c_i$ with the same local $\delta$-invariant as $\delta(c_i).$  Suppose that $\mathcal{C}'/S$ is locally versal, we denote the inverse image $j^{-1}(\prod_{i=1}^{n}B^{\delta_i}(c_i))$ in $S$ by $S^{\delta}.$

The natural map from the deformation space $B$ of $C$ to the deformation space $\prod_{i=1}^{n}B(c_i)$ of the singularities is smooth \cite[Section A]{eu}. By passing to Segre classes via Lemma \ref{lem 2.2}, we get $$m([C],B^{\delta})=s([C],B^{\delta})=s\big(\prod_{i=1}^{n}[c_i],\prod_{i=1}^{n}B^{\delta_i}(c_i)\big)= m([C],S^{\delta}).$$

The above equality shows that to determine $m([C],B^{\delta}), $ one may also work with locally versal families. From now on, the family $\mathcal{C}/B$ of curves is only assumed to be locally versal. The following lemma will be used later.

\begin{lemma}[{\cite[Lemma 4.2]{gv}}]\label{lem 3.1}
For any projective flat family of integral curves $\mathcal{C}^{H}\to H$ with at
worst planar singularities, there is a locally versal family of curves $\mathcal{C}\to B$
and a closed immersion $H\xhookrightarrow{}B$ such that $\mathcal{C}^H = \mathcal{C}\times_B H$.    
\end{lemma}
\subsection{Moduli of stable maps}
Turning now to moduli spaces of stable maps, recall that a genus $g_0$ stable map to a variety $X$ with class $\beta\in H_2(X,\mathbb{Z})$ is a morphism $$
\mu:\Sigma\To X
$$ such that $\Sigma$ is a connected projective nodal curve of arithmetic genus $g_0$, $\mu_*[\Sigma]=\beta,$ and the automorphism group of $\mu$ is finite. We are interested in the moduli $M_{g}(\mathcal{C}, [C])$ of genus $g$ stable maps to $\mathcal{C}$ with class $[C].$

By semicontinuity of the geometric genera, shrinking $B$ if necessary, we may assume that all fibres of the locally versal family $\mathcal{C}/B$ have geometric genus at least $g.$ Closed points of $M_{g}(\mathcal{C}, [C])$ are then in 1-1 correspondence with the normalizations
\begin{equation}\label{nor}
\Tilde{\mathcal{C}}_t \to \mathcal{C}_{t}\subset \mathcal{C}
\end{equation}
of curves over $t\in B^{\delta},$ since the images of the stable maps have geometric genus at least $g.$

The moduli space $M_{g}(\mathcal{C}, [C])$ is fine because all stable maps have only the trivial automorphism. There exists a universal stable map
\begin{equation*} 
p:\mathcal{U}\to  \mathcal
{C},
\end{equation*}
where $$\pi:\mathcal{U}\xrightarrow[]{}M_{g}(\mathcal{C}, [C])$$ is a smooth family of genus $g$ curves over $M_{g}(\mathcal{C}, [C]).$

Since we are considering the fibre curve class $[C]\in H_2(\mathcal{C},\mathbb{Z})$, the moduli spaces of stable maps behave well under restriction to subfamilies. 
\begin{lemma}\label{lem 3.3}
There exists a natural bijective map 
\begin{equation} \label{eq 5}
f:M_{g}(\mathcal{C}, [C])\To B^{\delta}\subset B
\end{equation}
sending the normalization \emph{(\ref{nor})} to the point $t\in B^{\delta}.$

Moreover, for any closed subscheme $Z\subset B$ that contains $[C]\in B,$ the moduli space of stable maps $M_{g}(\mathcal{C}\times_{B}Z, [C])$ is isomorphic to $$M_{g}(\mathcal{C}, [C])\times_B Z.$$ 
\end{lemma}
\begin{proof}
Consider the smooth universal family $
\pi:\mathcal{U}\to M_{g}(\mathcal{C}, [C]).$
By standard base-change arguments,  we have
\begin{equation}\label{eq 6}
\pi_*\mathcal{O}_{\mathcal{U}}=\mathcal{O}_{M_{g}(\mathcal{C}, [C])}.
\end{equation}

Assume first that $B$ is affine. From the map $\mathcal{U}\xrightarrow[]{p}  \mathcal
{C}\to B$ we get $$H^0(\mathcal{O}_B)\To H^0(\mathcal{O}_{\mathcal{U}})\cong H^0(\mathcal{O}_{M_{g}(\mathcal{C}, [C])}).$$ This
induces a map $f:M_{g}(\mathcal{C}, [C]) \to B$ 
that makes the following (non-Cartesian) diagram commute. For general $B,$ by considering an affine open cover of $B,$ we obtain such a map $f.$
$$
\begin{tikzcd}
 \mathcal{U} \arrow{r}{p} \arrow{d}{\pi} &  \mathcal{C} \arrow{d}{}  \\ 
 M_{g}(\mathcal{C}, [C]) \arrow[r, dashed, "f"] & B
\end{tikzcd}
$$

The map $
\mathcal{C}\times_{B}Z\xhookrightarrow{  } \mathcal{C}
$ induces a closed immersion 
$$M_{g}(\mathcal{C}\times_{B}Z, [C])\xhookrightarrow{\text{ \ \  }}  M_{g}(\mathcal{C}, [C]).$$ We now identify this closed subscheme $M_{g}(\mathcal{C}\times_{B}Z, [C])\subset M_{g}(\mathcal{C}, [C])$ with the closed subscheme $M_{g}(\mathcal{C}, [C])\times_B Z.$

The relation (\ref{eq 6}) and the above commutative diagram imply that for any closed subscheme $V\subset M_{g}(\mathcal{C}, [C]),$ the restriction $$V\subset M_{g}(\mathcal{C}, [C]) \xrightarrow[]{\text{ }f\text{ }} B$$ factors through the closed subscheme $Z\subset B$ if and only if the restriction 
$$
\pi^{-1}(V)\subset \mathcal{U} \To B
$$ factors through $Z.$
Note that $f^{-1}(Z)$ is the largest such $V$ so that $V \xrightarrow[]{} B$ factors through $Z,$ and $M_{g}(\mathcal{C}\times_{B}Z, [C])$ is the largest such $V$ so that $
\pi^{-1}(V) \to B
$ factors through $Z.$ They therefore coincide.
\end{proof}

The following lemma relates $B^{\delta}$ and $M_{g}(\mathcal{C}, [C]).$

\begin{lemma}[\cite{DH},\cite{eu}] \label{1}
Shrinking $B$ if necessary, $B^{\delta}$ is irreducible. The space $M_{g}(\mathcal{C}, [C])$ is a smooth variety, and the bijective morphism to $B^{\delta}$ in \emph{(\ref{eq 5})}
\begin{equation*}
f:M_{g}(\mathcal{C}, [C])\To B^{\delta}
\end{equation*}
is the normalization of $B^{\delta}$.   
\end{lemma}
\begin{proof}
The statements are  \cite[Proposition F.2]{eu} (see also \cite[Proposition 4.17]{DH}): while \cite[Proposition F.2]{eu} is stated for a semi-universal family $\mathcal{C}/B$, only the smoothness of $B\to\prod_{i=1}^{n}B(c_i)$ is used in the proof and the results hold for the locally versal families as well.
\end{proof}

\subsection{The multiplicity $m([C],B^{\delta})$}
Applying Lemma \ref{2} to the normalization $f$ in (\ref{eq 5}), we prove: 
\begin{thm}[Theorem \ref{theorem 1}]\label{thm 3.1}
The multiplicity of the point $[C]$ in $B^{\delta}$ satisfies:
$$
m([C],B^{\delta})\geq \len\bigl(M_{g}(C,[C])\bigr),
$$ and the equality holds if and only if $M_{g}(C,[C])$ is a local complete intersection.  
\end{thm}
\begin{proof}
Consider the bijective normalization map (\ref{eq 5}):
$$
f: M_{g}(\mathcal{C}, [C])\to B^{\delta}.
$$
The fat point $f^{-1}([C])$ is the moduli of stable maps 
$M_{g}(C,[C])$ by Lemma \ref{lem 3.3}.

Lemma \ref{lem 2.2} and \ref{2} give:
$$
m([C],B^{\delta})=s([C],B^{\delta})=s\bigl(f^{-1}([C]), M_{g}(\mathcal{C}, [C])\bigr).
$$

By \cite[Example 4.3.5 (c)]{ful}, 
$$
s\bigl(f^{-1}([C]), M_{g}(\mathcal{C}, [C])\bigr)\geq \len\bigl(f^{-1}([C])\bigr),
$$ and the equality holds if and only if $f^{-1}([C])=M_{g}(C,[C])$ is a local complete intersection in the smooth variety $M_{g}(\mathcal{C}, [C])$.  
\end{proof}

\begin{remark}\label{re}
In the case where $C$ is the compactification of $x^p=y^q$ in $\mathbb{C}^2$ which is smooth at $\infty$ and $p,q\in \mathbb{Z}_+,(p,q)=1,$ it is shown in \emph{\cite[Section G]{eu}} that $M_{0}(C, [C])$ is a local complete intersection, and indeed $
m([C],B^{\delta})= \len\bigl(M_{0}(C,[C])\bigr).
$
\end{remark}
Next, we prove Theorem \ref{thm 2}. 
\begin{thm}[Theorem \ref{thm 2}]\label{thm 3.2}
 Let $\mathcal{C}^{T}\to T$ be a flat family of integral planar curves. Suppose that $T$ and the relative compactified Jacobian $\Jacbar(\mathcal{C}^{T}/T)$ are smooth, and all fibres of $\mathcal{C}^{T}/T$ other than $\mathcal{C}^{T}_0=C$ have geometric genus greater than $g$. Then 
\begin{equation*} 
m([C], B^{\delta})=\len(M_{g}(\mathcal{C}^{T},[C])).
\end{equation*}    
\end{thm}
By Lemma \ref{lem 3.1}, given such a family $\mathcal{C}^{T}\to T$, we may assume $\mathcal{C}^{T}/T$ is pulled back from the locally versal family $\mathcal{C}/B$ via a closed immersion $T\xhookrightarrow{} B.$  Moreover, the desired base $T$ always exists since we can take it to be any generic $\delta(C)$-dimensional smooth subvariety of $B,$ by the smoothness criterion \cite[Corollary B.3 and Proposition C.5]{eu} for the relative compactified Jacobians. 
\begin{proof}
Recall from  \cite[Theorem D.2(3)]{eu} that $B^{\delta}$ has codimension $\delta(C)$ in $B$. The smooth subspace $T\subset B$ has dimension at most $\delta(C)$ since it intersects $B^{\delta}$ only at the point $[C].$

Moreover, since the relative compactified Jacobian $\Jacbar(\mathcal{C}^{T}/T)$ is smooth, $T$ is transverse to the the support of the tangent cone of $B^{\delta}$ at $[C]$ by \cite[Corollary B.3 and Proposition C.5]{eu}. This support is a codimension $\delta(C)$ linear subspace of $T_{[C]}B$ by \cite[Theorem D.2]{eu}, therefore $T$ has dimension at least and hence exactly $\delta(C).$

We claim that $m([C], B^{\delta})$ is the intersection multiplicity $i([C], B^{\delta}\cdot T; B).$ Indeed, inside $T_{[C]}B,$ the tangent space of $T$ and the tangent cone of $B^{\delta}$ intersect only at $0$, and their projectivisations are disjoint in the projectivisation of $T_{[C]}B$. The formula in \cite[Theorem 12.4(a)]{ful} then shows  
\begin{equation}\label{eq7}
i([C], B^{\delta}\cdot T; B)=m([C],T)\times m([C],B^{\delta})=m([C],B^{\delta}).
\end{equation}

Because $T$ is smooth, $T\xhookrightarrow{}B$ is a regular embedding and 
\begin{equation}
i([C], B^{\delta}\cdot T; B)=s(B^{\delta}\times_{B}T,B^{\delta}).
\end{equation}

Consider the birational morphism $f$ in (\ref{eq 5}): $$
f: M_{g}(\mathcal{C}, [C])\To B^{\delta}.
$$ 
By the birational invariance of Segre classes (Lemma \ref{2}), 
\begin{equation}
s(B^{\delta}\times_{B}T,B^{\delta})=s(f^{-1}(B^{\delta}\times_{B}T),M_{g}(\mathcal{C}, [C])).
\end{equation}

Since $f$ is bijective, $f^{-1}(B^{\delta}\times_{B}T)=M_{g}(\mathcal{C}, [C])\times_B T$ is a fat point. Since $T\xhookrightarrow{}B$ is a regular embedding, $T$ is locally cut out in $B$ by $\text{dim}_{\mathbb{C}}B^{\delta}$ regular functions. Then the fat point $f^{-1}(B^{\delta}\times_{B}T)$ in the smooth $M_{g}(\mathcal{C}, [C])$ is also cut out by $\text{dim}_{\mathbb{C}}B^{\delta}=\text{dim}_{\mathbb{C}}M_{g}(\mathcal{C}, [C])$ functions,  and 
the normal cone of $f^{-1}(B^{\delta}\times_{B}T)\subset M_{g}(\mathcal{C}, [C])$ is a trivial vector bundle. We see \begin{equation}
s(f^{-1}(B^{\delta}\times_{B}T),M_{g}(\mathcal{C}, [C]))=\len(f^{-1}(B^{\delta}\times_{B}T)).
\end{equation}

Combining the above, we arrive at the equality 
\begin{equation}
m([C],B^{\delta})=\len(f^{-1}(B^{\delta}\times_{B}T)).
\end{equation}
To conclude, we note that $f^{-1}(B^{\delta}\times_{B}T)=M_{g}(\mathcal{C}, [C])\times_B T$ is isomorphic to $M_{g}(\mathcal{C}\times_{B}T, [C])=M_{g}(\mathcal{C}^{T}, [C])$ by Lemma \ref{lem 3.3}. 
\end{proof}
\begin{remark}
In light of Lemma \emph{\ref{lem 2.1}}, the equality \emph{(\ref{eq7})} also shows that the intersection multiplicity $i([C], B^{\delta}\cdot T; B)$ is minimal for a $\delta(C)$-dimensional subvariety $T\subset B.$
\end{remark}
\begin{corollary}[Corollary \ref{coro 1}]\label{coro 3.1}
Suppose $C$ is a rational curve on a K3 surface $S.$ The moduli space of stable maps $M_0(S,[C])$ contains a connected component $M_0(S,[C])_{[f_C]}$ with a unique closed point $[f_C:\mathbb{P}^1\to S].$ We have
\begin{equation*}
\chi(\Jacbar_C)=\len(M_0(S,[C])_{[f_C]}).
\end{equation*}
\end{corollary}
\begin{proof}
Since $S$ has Kodaira dimension 0, there are only finitely many rational curves in the linear system $|C|.$ Therefore $[C]$ has an open neighbourhood $T\subset |C|$ which parametrises only integral curves on $S$ and no rational curves other than $C.$ This shows the existence of the component $M_0(S,[C])_{[f_C]}.$

Denote the universal curve over $T$ by $\mathcal{C}^{T}.$ By \cite[Section 5.3]{gv}, the relative compactified Jacobian $\Jacbar(\mathcal{C}^{T}/T)$ is isomorphic to an open subset of the moduli of stable one-dimensional sheaves on $S$ with $ch_2=[C].$ Therefore $\Jacbar(\mathcal{C}^{T}/T)$ is smooth, and we may take the desired family in Theorem \ref{thm 3.2} to be $\mathcal{C}^{T}/T.$

We finish by constructing maps between $M_0(\mathcal{C}^{T},[C])$ and $M_0(S,[C])_{[f_C]},$ which are inverse to each other. The composition $$
\mathcal{C}^{T}\subset S\times T\To S
$$ induces a map 
\begin{equation}\label{map 1}
M_0(\mathcal{C}^{T},[C]) \To  M_0(S,[C])_{[f_C]}.
\end{equation}

For the reverse direction, first note that the stable map $f_C:\mathbb{P}^1\to S$ has only the trivial automorphism. There then exists a universal family of stable maps $$
\nu:\mathbb{U}\To M_0(S,[C])_{[f_C]}\times S,
$$ where $\mathbb{U}\to M_0(S,[C])_{[f_C]}$ is a smooth family of rational curves.

The map $\nu$ is finite, and in particular affine. Pushing forward the structure sheaf via $\nu,$ $\nu_*\mathcal{O}_{\mathbb{U}}$ is a flat $M_0(S,[C])_{[f_C]}$-family of sheaves on $S.$ This induces a map $$
M_0(S,[C])_{[f_C]}\To \Jacbar(\mathcal{C}^{T}/T),
$$ since $\Jacbar(\mathcal{C}^{T}/T)$ is isomorphic to an open subset of the moduli of stable sheaves on $S.$ Consider the composition $$
M_0(S,[C])_{[f_C]}\To \Jacbar(\mathcal{C}^{T}/T)\To T. 
$$ This makes $M_0(S,[C])_{[f_C]}$ a $T$-scheme. We claim that the stable map $\nu$ factors as 
\begin{equation}\label{eq 11}
\nu:\mathbb{U}\To M_0(S,[C])_{[f_C]}\times_T \mathcal{C}^{T}\xhookrightarrow{\text{ \ \ }} M_0(S,[C])_{[f_C]}\times S,
\end{equation}
giving the desired 
\begin{equation}\label{map 2}
M_0(S,[C])_{[f_C]}\To M_0(\mathcal{C}^{T},[C]).
\end{equation} Indeed, the sheaf $\nu_*\mathcal{O}_{\mathbb{U}}$ is supported on $M_0(S,[C])_{[f_C]}\times_T \mathcal{C}^{T}$ as it is pulled back from the universal sheaf on $\Jacbar(\mathcal{C}^{T}/T)\times_T \mathcal{C}^{T},$ the factorisation (\ref{eq 11}) follows since $\nu$ is affine.

The maps (\ref{map 1}), (\ref{map 2}) between the two moduli spaces $M_0(S,[C])_{[f_C]}$ and $M_0(\mathcal{C}^{T},[C])$ are inverse to each other, since at the level of universal objects they exchange the corresponding universal stable maps. 
\end{proof}

\printbibliography
\end{document}